\documentclass[11pt]{amsart}
\usepackage{amsmath}
\usepackage{amsxtra}
\usepackage{amscd}
\usepackage{amsthm}
\usepackage{amsfonts}
\usepackage{amssymb}
\usepackage{mathrsfs}
\usepackage[all]{xy}
\usepackage[pdftex]{color,graphicx}

\DeclareMathOperator{\PSL}{PSL} \DeclareMathOperator{\PGL}{PGL}
\DeclareMathOperator{\PSU}{PSU} 
 \DeclareMathOperator{\GL}{GL}
\DeclareMathOperator{\Sz}{Sz} \DeclareMathOperator{\End}{End}
\DeclareMathOperator{\Aut}{Aut}

\DeclareMathOperator{\der}{der} \DeclareMathOperator{\Frob}{Frob}
\DeclareMathOperator{\rk}{rk} \DeclareMathOperator{\Epi}{Epi}

\theoremstyle{plain}
\newtheorem{theorem}{Theorem}[section]
\newtheorem{prop}[theorem]{Proposition}

\newtheorem{thm}[theorem]{Theorem}
\newtheorem{lem}[theorem]{Lemma}
\newtheorem{cor}[theorem]{Corollary}
\newtheorem{rem}[theorem]{Remark}
\newtheorem{defn}[theorem]{Definition}

\theoremstyle{definition}

\newtheorem{conj}{Conjecture}

\newtheorem*{example}{Example}
\theoremstyle{remark}

\numberwithin{equation}{section}

\newcommand{\al}{\alpha}
\newcommand{\be}{\beta}
\newcommand{\ga}{\gamma}
\newcommand{\de}{\delta}

\newcommand{\Ga}{\Gamma}
\newcommand{\De}{\Delta}

\newcommand{\bF}{\mathbb F}

\newcommand{\qq}{\quad}

\newcommand{\bmat}{\left ( \begin{matrix} }
\newcommand{\emat}{\end{matrix} \right ) }

\newcommand{\lbl}[1]{\label{#1}}

\title[Connectivity of the PRA Graph of Simple Groups of Bounded Lie Rank]
{Connectivity of the Product Replacement Graph of Simple Groups of Bounded Lie Rank}
\author{Nir Avni, Shelly Garion}
\date{\today}

\begin{document}

\maketitle

\begin{abstract}
The \emph{Product Replacement Algorithm} is a practical algorithm for
generating random elements of a finite group. The algorithm can be
described as a random walk on a graph whose vertices are the
generating $k$-tuples of the group (for a fixed $k$).

We show that there is a function $c(r)$ such that for any finite simple group of Lie type, with Lie rank $r$, the Product Replacement Graph of the generating $k$-tuples is connected for any $k \geq c(r)$.

The proof uses results of Larsen and Pink~\cite{LP} and does not rely on the classification of finite simple groups.
\end{abstract}

\footnotetext{2000 {\it Mathematics Subject Classification:} 20D06, 20G40, 20D60. }

\section{Introduction}
\subsection{The Product Replacement Algorithm and its Graph}
The \emph{Product Replacement Algorithm (PRA)} is a practical
algorithm for generating random elements of a finite group. The algorithm was introduced
and analyzed in~\cite{CLMNO}. Although it has no rigorous
justification, practical experiments showed excellent performance.
It quickly became a popular algorithm for
generating random group elements, and was included in two frequently used group algebra packages: GAP and MAGMA. The algorithm
has been previously studied in~\cite{BP,GaP,LP1,P}.

The Product Replacement Algorithm can be described as a random walk
on a graph, called the {\em Product Replacement Graph} (or the {\em
PRA Graph}). It will be more convenient for us to look at the
following extended graph. Let $\Ga$ be a finite group and let
$d(\Ga)$ be the minimal number of generators of $\Ga$. We denote the
group generated by a set $S$ by $\langle S \rangle$. For any $k \geq
d(\Ga)$, let
\[
    V_k(\Ga)=\{(g_1,\ldots,g_k) \in \Ga^k: \langle g_1,\ldots,g_k
    \rangle = \Ga\}
\]
be the set of all \emph{generating $k$-tuples} of $\Ga$.

The \emph{extended PRA graph}, denoted by $\tilde X_k(\Ga)$, has
$V_k(\Ga)$ as its set of vertices. The edges of the extended PRA graph correspond to the
following so-called \emph{Nielsen moves} $R_{i,j}^{\pm}$,
$L_{i,j}^{\pm}$, $P_{i,j}$, $I_{i}$ for $1 \leq i \neq j \leq k$,
where
\begin{align*}
    R_{i,j}^{\pm} &:  (g_1,\ldots,g_i,\ldots,g_k) \rightarrow
    (g_1,\ldots,g_i\cdot g_j^{\pm 1},\ldots,g_k) \\
    L_{i,j}^{\pm} &:  (g_1,\ldots,g_i,\ldots,g_k) \rightarrow
    (g_1,\ldots,g_j^{\pm 1} \cdot g_i,\ldots,g_k) \\
    P_{i,j} &:  (g_1,\ldots,g_i,\ldots,g_j,\ldots,g_k) \rightarrow
    (g_1,\ldots,g_j,\ldots,g_i,\ldots,g_k)  \\
    I_{i} &:  (g_1,\ldots,g_i,\ldots,g_k) \rightarrow
    (g_1,\ldots,g_i^{-1},\ldots,g_k)
\end{align*}

Strictly speaking, the Product Replacement Algorithm is a random
walk on a subgraph $X_k(\Ga)$ of $\tilde X_k(\Ga)$, which is obtained by
removing the edges corresponding to the $P_{i,j}$ and $I_{i}$. The
output of the algorithm is a random element chosen from the tuple at
the end of the random walk. As observed in \cite[Proposition
2.2.1]{P}, when $k\geq d(\Ga)+1$, the graph $X_k(\Ga)$ is connected
if and only if $\tilde X_k(\Ga)$ is connected.

\subsection{$T$-Systems}
There is another point of view on the PRA graph, which was,
historically, the motivation for its study, even before
\cite{CLMNO}. Given a finite group $\Ga$ and a fixed integer $k \geq
d(\Ga)$, a {\em presentation} of $\Ga$ by $k$ generators is an
epimorphism $\phi: F_k \twoheadrightarrow \Ga$, where $F_k$ is the
free group on $k$ generators\footnote{The standard definition of a
presentation of $\Ga$ is an epimorphism $\phi$ as above, together
with a generating set for the kernel of $\phi$. We will, however,
identify two presentations if they have the same kernel.}. One can
identify the set of epimorphisms $\Epi(F_k \twoheadrightarrow \Ga)$
with $V_k(\Ga)$. The group $\Aut(F_k) \times \Aut(\Ga)$ acts on
$\Epi(F_k\twoheadrightarrow\Ga)$ by $(\tau, \sigma): \phi \mapsto
\sigma \circ \phi \circ \tau^{-1}$, where $\tau \in \Aut(F_k),
\sigma \in \Aut(\Ga)$ and $\phi \in \Epi(F_k \twoheadrightarrow
\Ga)$. The question whether this action is transitive was raised by
B.H. Neumann and H. Neumann in \cite{NN} and studied further in
\cite{Du1,Du2,E,Gi,GP,Ne,P}. An orbit of $\Aut(F_k)\times\Aut(\Ga)$
in $V_k(\Ga)$ is called a {\em system of transitivity}, and also
{\em $T$-system} or {\em $T_k$-system}, when we specify the value of
$k$.

It is well-known that $\Aut(F_k)$ is generated by the Nielsen moves
$\{ R_{i,j}^{\pm},L_{i,j}^{\pm},P_{i,j},I_i \}_{1 \leq i,j \leq k}$,
viewing them as automorphisms of $F_k$. Thus there is a map from the
connected components of $\tilde X_k(\Ga)$ to the set of
$T_k$-systems. As observed in \cite[Proposition 2.4.1]{P}, when $k
\geq 2d(\Ga)$, the graph $\tilde X_k(\Ga)$ is connected if and only
if $\Ga$ has only one $T_k$-system.

\subsection{Connectivity of the PRA Graph}
The purpose of this paper is to study the connectivity of the
extended PRA graph $\tilde X_k(\Ga)$, where $\Ga$ is a finite simple
group of Lie type. We start with some historical background.

Let $\Ga$ be a finite group, and let $k \geq d(\Ga)$ be an integer.
There are several examples of groups $\Ga$ such that for $k=d(\Ga)$,
the graph $\tilde X_k(\Ga)$ is not connected (see
\cite{Du1,GS,GP,Ne,P}). However, there are no known examples of
groups $\Ga$ such that $\tilde{X}_k(\Ga)$ is disconnected for $k
\geq d(\Ga)+1$. Moreover, Dunwoody~\cite{Du2} has proved that for a
finite solvable group $\Ga$, the graph $\tilde X_k(\Ga)$ is
connected whenever $k \geq d(\Ga)+1$. The following conjecture
arises.

\begin{conj} \label{conj_gen}
For all finite groups $\Ga$, if $k\geq d(\Ga)+1$, then $\tilde X_k(\Ga)$ is connected.
\end{conj}

A particularly interesting case is when the group $\Ga$ is a finite
simple group. In this case, $d(\Ga)=2$, and one can reformulate
Conjecture~\ref{conj_gen} as follows.
\begin{conj}[Wiegold] \label{conj:Weigold}
If $\Ga$ is a finite simple group and $k\geq 3$, then $\tilde
X_k(\Ga)$ is connected.
\end{conj}

This conjecture was originally stated by Wiegold for $T$-systems (see \cite[Conjecture 1.2]{E}),
and it has been verified only in the few following cases.

\begin{theorem} \label{simple_conn}
$\tilde X_k(\Ga)$ is connected in the following cases:
\begin{enumerate}
\renewcommand{\theenumi}{\alph{enumi}}
\item \label{t:PSLp} \cite{Gi}.
$\Ga=\PSL(2,p)$, where $p\geq 5$ is prime and $k \geq 3$.
\item \label{t:PSL2} \cite{E}.
$\Ga=\PSL(2,2^m)$, where $m\geq 2$ and $k \geq 3$.
\item \label{t:PSLq} \cite{Ga}.
$\Ga=\PSL(2,q)$, where $q=p^e$ is an odd prime power and $k \geq 4$.
\item \label{t:Su} \cite{E}.
$\Ga=\Sz(2^{2m+1})$, where $m\geq 2$ and $k \geq 3$.
\item \label{t:An} \cite{CP, Da}.
$\Ga=A_n$, where $6 \leq n \leq 11$ and $k=3$.
\end{enumerate}
\end{theorem}

In this paper we prove the existence of a function $c(r)$ such that
if $\Ga$ is a finite simple group of Lie type with Lie rank $r$,
then the graph $\tilde X_k(\Gamma)$ is connected for all $k \geq
c(r)$.

We note that for such a group $\Gamma$ of Lie type, whose rationality
field has $q = p^e$ elements, one can deduce the existence of a
constant $C$, such that the graph $\tilde X_k(\Gamma)$ is connected
when $k \geq C r^2 e$, by invoking results of~\cite{Ni} and
~\cite[Theorem 0.2]{LP}. The latter bound, however, depends on the
size of the defining field of $\Gamma$, while our bound  depends only on the Lie rank of $\Gamma$; it is
independent on the size of the defining field of $\Gamma$.
Unfortunately, the bound we obtain is at least exponential in the
Lie rank $r$.

\subsection{Main Theorem}
We formulate our result more precisely. For a group $H$, we denote
the center of $H$ by $Z(H)$, and denote the first derived subgroup
of $H$ by $H^{\der}$. Let $G_r$ be a (possibly twisted) simple Lie
type of rank $r$. Recall that if $\De = G_r(q)$ is a finite simple
group of Lie type over a field of size $q$, then any
perfect\footnote{A group $H$ is called {\em perfect} if $H=H^{\der}$.}
central extension $\Ga$ of $\De$ is called a finite {\em
quasi-simple} group of Lie type $G_r$. The {\em Lie rank} of $\Ga$
is defined to be the Lie rank of $\De$. The Lie type and Lie rank of
$\Ga$ are easily computed from $\De$, as $\De=\Ga/Z(\Ga)$. For an
introduction to the finite simple groups of Lie type (and much
more) see \cite{Ca2} and \cite{Ca1}.

\begin{thm} \lbl{thm:main.theorem} For every $r$ there is a constant $c(r)$ such that if $\Ga$ is a finite, quasi-simple group of Lie type with rank at most $r$, then for every integer $k \geq c(r)$, the extended Product Replacement Graph $\tilde X_k(\Ga)$ is connected.
\end{thm}

The main step in the proof is to show that if $k$ is large enough,
then any generating $k$-tuple $(g_1,\dots,g_k)$ can be connected by
Nielsen moves to a redundant one, i.e., a tuple such that one of its coordinates is the identity element of $\Ga$.

\subsection{Organization}
For the convenience of the reader, we describe the organization of
the paper, as well as give a bird's eye view of the proof. In
Section 2, we collect facts about subgroups of finite simple groups
of Lie type. In Section 3 we show that there is a constant
$c_1=c_1(r)$, such that for any $k\geq c_1$, every generating tuple
$(g_1,\ldots ,g_k)$ can be connected to a tuple $(g_1',\ldots
,g_k')$, and the subgroup $\De_0$, generated by
$g_1',\dots,g_{c_1}'$, is also a quasi-simple group of the same Lie
type as $\Ga$ (possibly over a smaller field). This is done by using
\cite{LP}.

Subgroups of $\Ga$ that contain $\De_0$ are highly restricted.
Avoiding some delicate issues, if $\De_0=G_r(q)$ and $\Ga=G_r(q^e)$,
then a subgroup of $\Ga$ containing $\De_0$ is of the form
$G_r(q^d)$, for some $d$ dividing $e$. Using this fact, we prove in
Section 4 that there is a constant $c_2=c_2(r)$ such that, our
generating tuple can be connected to a generating tuple
$(g_1'',\ldots ,g_k'')$, and the subgroup generated by
$g_1'',\dots,g_{c_1}'',\dots,g_{c_1+c_2}''$ is, in fact, equal to
$\Ga$. This way we get the desired redundant generators. Our method
here is a generalization of \cite{E}.

The proof is slightly simpler in the case that the characteristic is different from 2 and 3. We assume this throughout the paper. In the appendix, we show how to adapt the proof when the characteristic is equal to 2 or 3.

\subsection*{}{\em Acknowledgments.}
This paper is part of the authors' Ph.D. studies under the
guidance of Alex Lubotzky, whom we would like to thank for
support and good advice. We are also grateful to Michael Larsen,
Tsachik Gelander and Aner Shalev for many useful discussions, and to the referee for his remarks on this article.

The second author acknowledges the support of the Israeli Ministry of Science, Culture, and Sport.

\section{Sufficiently General Subgroups of Finite Simple Groups of Lie Type}
\subsection{Notations} Let $p$ be a prime number and $q=p^e$ be a power of $p$. We denote the finite field of size $q$ by $\bF _q$. The algebraic closure of $\bF _p$ (which is equal to the algebraic closure of $\bF _q$) will be denoted by $\overline{\bF _p}$. For every $n$ we denote the set of invertible $n$ by $n$ matrices with entries in a field $\bF$ by $\GL_n(\bF )$.

Let $G \subset \GL_n(\overline{\bF _p})$ be a simple, connected and adjoint algebraic group defined over $\bF _q$. Starting from Subsection \ref{subsec:F.Ga}, we shall assume that $p\neq2,3$. For every $f$, we denote the intersection $G \cap \GL_n(\bF _{q^f})$ by $G(\bF _{q^f})$.
Excluding finitely many cases, the subgroup $G(\bF _q)^{\der}$ is a simple group (this follows from \cite[\S 2.9]{Ca2} and \cite[\S 11.1]{Ca1}).

The construction above gives many of the finite simple groups of Lie type, but not all of them. In order to get all finite groups of Lie type, we make the following definition:
\begin{defn}\lbl{defn:group.Lie.type}\cite[\S 1.17]{Ca2}. Let $p$ be a prime number and let $n$ be a natural number.
 \begin{enumerate}
\item For every power $q$ of $p$, the {\em standard Frobenius map} of $\bF _q$ is the function
\[
\Frob_q :\GL_n(\overline{\bF _p})\to \GL_n(\overline{\bF _p})
\]
given by
\[
(x_{i,j}) \mapsto (x_{i,j}^q) .
\]
\item Let $G\subset \GL_n(\overline{\bF _p})$ be an algebraic group. A homomorphism $F:G \to G$ is called a {\em Frobenius map} if there are natural numbers $k$,$m$ and $e$, and a faithful representation $\rho :G \to \GL_m(\overline{\bF _p})$, such that $F^k$ is the restriction of the standard Frobenius map $\Frob_{\bF _{p^e}} : \GL_m(\overline{\bF _p}) \to \GL_m(\overline{\bF _p})$ to $\rho (G)$. In this case we define $q_F = p^{e/k}$.
\item Let $G\subset \GL_n(\overline{\bF _p})$ be an algebraic group and let $F:G\to G$ be a Frobenius map. We denote
\[
G^F = \{ g\in G | F(g)=g \} .
\]
\item A finite simple group of Lie type is a simple group of the form $(G^F)^{\der}$, where $G$ is a connected, simple and adjoint algebraic group, and $F$ is a Frobenius map.
\end{enumerate}
\end{defn}

\begin{example} \begin{enumerate}
\item For every $m$, and every prime power $q$, there are only finitely many fixed points of $\Frob_q$ in $\GL_m(\overline{\bF_p})$. Therefore, the same is true for any Frobenius map. Since $F$ is a homomorphism, we get that  $G^F$ is always a finite group.
\item If $G\subset \GL_n(\overline{\bF_q})$ is an algebraic group defined over $\bF_q$, and if $F=\Frob_q$ is the Frobenius map that corresponds to the inclusion $\rho :G \hookrightarrow \GL_n$, then $G^F = G(\bF _q)$.
\item As an example of a finite simple group of Lie type that is not obtained from the standard Frobenius map, one can check that if $G=\PSL_n(\overline{\bF_p})$, then the map $F$ defined by
\[
g \mapsto \Frob_p ((g^{T})^{-1})
\]
is a Frobenius map. The group of fixed points of $F$ is called the Unitary (or twisted $A_{n-1}$) group, and is denoted by $\PSU_n(p)$.
\end{enumerate}
\end{example}

\begin{rem} \lbl{rem:size.field} If $G^F=G(\bF_q)$, then $q_F$ is equal to the size of the field $\bF_q$. In the twisted case, the group $G^F$ does not have a field of rationality, but the number $q_F$ plays the role of the size of the field of rationality. For example, the size of $G^F$ is approximately $q_F^{\dim G}$.
\end{rem}

\subsection{Sufficiently General Subgroups}\lbl{defn.suff.gen}

We recall the definition in \cite{LP} of a sufficiently general subgroup of a simple algebraic group.
Let $G$ be a simple, connected and adjoint algebraic group. For every finite dimensional, rational representation
\[
\rho : G \to \GL_d(\overline{\bF _p}),
\]
and for every finite subgroup $\Ga\subset G$, we say that $\Ga$ is {\em sufficiently general with respect to $\rho$}, if every linear subspace of $\overline{\bF_p}^d$ which is $\rho(\Ga)$-invariant, is also $\rho(G)$-invariant.

In \cite{LP}, the authors construct a certain finite dimensional, rational representation
\[
\rho _G : G \to \GL_d(\overline{\bF _p})
\]
whose dimension $d$ depends only on the Lie type of $G$ (and not on the characteristic), such that the following theorem is true:

\begin{thm} \lbl{thm:LP} \cite[Theorem 0.6]{LP}. Let $G$ be a simple, connected and adjoint algebraic group over a field of positive characteristic. If $\Ga \subset G$ is a finite subgroup which is sufficiently general with respect to $\rho_G$, then there is a Frobenius map $F:G\to G$ such that $(G^F)^{\der}$ is a simple group, and
\[
(G^F)^{\der} \subset \Ga \subset G^F .
\]
\end{thm}

\begin{rem} \lbl{rem:schur.mult} It is a well known fact (and it follows from \cite[Theorems 9.4.10 and 14.3.1]{Ca1}, together with Remark \ref{rem:size.field}) that the index $[G^F:(G^F)^{\der}]$ is less than or equal to twice the rank of $G$.
\end{rem}

From this point on and until Remark \ref{rem:new.suff.gen}, we shall say that a finite subgroup is sufficiently general (without giving a representation) if it is sufficiently general with respect to the representation $\rho_G$.

\subsection{The Field $\bF _\Ga$} \lbl{subsec:F.Ga}
We describe here in more detail the construction of the Frobenius map in Theorem \ref{thm:LP}. In this section we shall assume that $p$ is not equal to 2 or 3.

\begin{defn} Let $G$ be a simple algebraic group. A {\em minimal unipotent} is a non-identity element in the center of the unipotent radical of a Borel subgroup.
\end{defn}

Let $G$ be a simple, connected and adjoint algebraic group. Let $\De \subset G$ be a sufficiently general subgroup. Corollary 8.11 of \cite{LP} implies that there is a minimal unipotent $v\in \De$. Let $V\subset G$ be the one parameter subgroup that contains $v$. The ring $\End (V)$ is isomorphic to the field $\overline{\bF_p}$. Consider $N_G(V)$, the set of elements $g\in G$ that normalize $V$. We have an obvious map $N(V)\to \End (V)$.

\begin{defn} \lbl{defn:F.De} $\bF _\De$ is the subring of $\End (V)$ generated by the image of $\De \cap N(V)$. Denote the size of $\bF_\De$ by $q_\De$.
\end{defn}

A-priori, the ring $\bF_\De$ and the integer $q_\De$ depend on the choice of $v$. However, as follows from the relation (\ref{eq:double.containment}) below, they are determined by $\De$ alone.

As a finite subring of a field, $\bF _\De$ is itself a field. Moreover, $\De \cap V$ is a one dimensional vector space over $\bF _\De$.

\begin{thm} \cite[Theorem 9.1]{LP}. Let $G$ be a simple, connected and adjoint algebraic group. Let $\De$ be a sufficiently general subgroup, and fix a minimal unipotent $v\in \De$. Then there is an $\overline{\bF_p}$-vector space $M$, an element $m_0\in M$, and a representation $\sigma :G\to \Aut (M)$ such that for every subgroup $\De \subset \Ga \subset G$, the $\bF _\Ga$ submodule
\[
M_0 ^\Ga =\bF _\Ga -span \{\sigma (\Ga )m_0 \}
\]
satisfies that the natural map
\[
M_0 \otimes _{\bF _\Ga} \overline{\bF_p} \to M
\]
is an isomorphism.
\end{thm}

Given $M,\sigma,\De$ as above, for every subgroup $\De \subset \Ga
\subset G$ we have the standard Frobenius map $\Frob_\Ga :\Aut(M)\to
\Aut(M)$ relative to the field $\bF _\Ga$. By the proof of Lemma 9.4
in \cite{LP}, there is a Frobenius map $F_\Ga :G\to G$ such that
$\sigma \circ F_\Ga = \Frob_\Ga \circ \sigma$. As shown in the proof
of Theorem 0.5 of \cite{LP}, such $F_\Ga$ satisfies
\begin{equation} \lbl{eq:double.containment}
(G^{F_\Ga})^{\der} \subset \Ga \subset G^{F_\Ga} .
\end{equation}

\subsection{Properties of Sufficiently General Subgroups}

\begin{prop} \lbl{prop:pair.suff.gen} Let $G$ be a simple, connected and adjoint algebraic group.
\begin{enumerate}
\item If $\De \subset \De' \subset G$ are sufficiently general subgroups of $G$, then there is a Frobenius map $F:G\to G$ and an integer $m$ such that
\[
(G^F)^{\der} \subset \De \subset G^F
\]
and
\[
(G^{F^m})^{\der} \subset \De' \subset G^{F^m} .
\]
\item If $\De \subset G$ is sufficiently general, $\Ga _1,\Ga _2$ are finite subgroups of $G$ that contain $\De$, and $\bF_{\Ga _1}=\bF_{\Ga _2}$, then there is a Frobenius map $F:G\to G$ such that
\[
(G^F)^{\der} \subset \Ga _1,\Ga_2 \subset G^F .
\]
\end{enumerate}
\end{prop}

\begin{proof}
\begin{enumerate}
\item From Definition \ref{defn:F.De}, it is clear that if $\De \subset \De' \subset G$ are sufficiently general, then $\bF _{\De} \subset \bF _{\De'}$. Therefore $\Frob_{\De'}$ is a power of $\Frob_{\De}$, and we can take $F_{\De'}$ to be the same power of $F_{\De}$.

\item By $(1)$, there is a Frobenius map $F:G\to G$ and integers $n_1,n_2$ such that
\[
(G^{F^{n_1}})^{\der} \subset \Ga _1 \subset G^{F^{n_1}} \quad \textrm{and} \quad (G^{F^{n_2}})^{\der} \subset \Ga _2 \subset G^{F^{n_2}} .
\]
Since $\bF_{\Ga_1}=\bF_{\Ga_2}$, it follows that $n_1=n_2$.
\end{enumerate}
\end{proof}

\begin{prop} \lbl{prop:big.is.general} \cite[Proposition 3.5]{LP}. Let $G$ be a simple, connected and adjoint algebraic group. There is a constant $q_0$ that depends only on the Lie rank of $G$ such that for every Frobenius map $F:G\to G$, the condition $q_F > q_0$ implies that the subgroup $(G^F)^{\der}$ is sufficiently general.
\end{prop}

\begin{prop} \lbl{prop:conj.classes} \cite[Theorem 6.6]{LP}. Let $G$ be a simple, connected and adjoint algebraic group, and let $\De \subset G$ be a sufficiently general subgroup. Then there is a constant $a_0$, that depends only on the Lie rank of $G$, such that for every $g \in \De$, the number of $\De$-conjugacy classes in $\De \cap g^G$ is less than $a_0$.
\end{prop}

\subsection{Regular Semisimple Elements in Sufficiently General Groups}

\begin{lem} Let $G$ be a simple algebraic group. There is a rational representation $\rho _{reg}: G \to \GL_{d}(\overline{\bF_p})$, whose dimension $d$ depends only on the Lie rank of $G$, such that the following holds: If $\De \subset G$ is a subgroup, and every linear subspace of $\overline{\bF_p}^d$ that is $\rho_{reg}(\De)$-invariant is also $\rho_{reg}(G)$-invariant, then for every $g_0\in G$, the set $g_0\De$ contains a regular semisimple element.
\end{lem}

\begin{proof} Fix an embedding of $G$ into matrices of size $n\times n$. This can be done such that $n$ depends only on the Lie rank of $G$. Using this embedding, we think of the coordinate ring of $G$ as a quotient of the coordinate ring of $M_n$. There is an element $P(x)$ in the coordinate ring of $G$ such that for all $g\in G$, the condition $P(g)\neq 0$ implies that $g$ is regular semisimple. Pick an element $\widetilde{P}(x)$ of the coordinate ring of $M_n$ that lie over $P$. We denote the degree of $\widetilde{P}(x)$ by $D$.

Let $\overline{\bF_p}[M_n]^{\leq D}$ be the linear space of polynomial functions in $\overline{\bF_p}[M_n]$ of degree less than or equal to $D$, and let $\overline{\bF_p}[G]^{\leq D}$ be the image of $\overline{\bF_p}[M_n]^{\leq D}$ in the coordinate ring of $G$. Let $\rho _{reg}$ be the representation of $G$ on $\overline{\bF_p}[G]^{\leq D}$, where $g$ acts on a polynomial $Q(x)$ by right translations:
\[
(\rho_{reg}(g)Q)(x)=Q(x\cdot g) .
\]
For every $g_0\in G$, and every non-trivial $Q(x)\in \overline{\bF _p}[G]^{\leq D}$, there is $g\in G$ such that
\[
(\rho_{reg}(g)Q)(g_0) = Q(g_0\cdot g) \neq 0.
\]
Therefore every non-trivial and $\rho_{reg}(G)$-invariant subspace of $\overline{\bF _p}[G]^{\leq D}$ contains a polynomial $Q(x)$ such that $Q(g_0)\neq 0$.

Assume that $g_0\in G$, that $\De \subset G$, and that every $\rho_{reg}(\De)$-invariant subspace is also $\rho_{reg}(G)$-invariant. Consider the subspace $W\subset \overline{\bF _p}[G]^{\leq D}$ spanned by the set $\rho_{reg}(\De)(P(x))$. It is clearly $\rho(\De)$-invariant, and so, by our assumption, it is $\rho(G)$-invariant. Therefore $W$ contains a polynomial $Q(x)$ such that $Q(g_0)\neq 0$, and hence there is a $\de \in \De$ such that $P(g_0 \cdot \de)\neq 0$. The element $g_0\cdot \de$ is regular semisimple.
\end{proof}

\begin{rem} \lbl{rem:new.suff.gen} Denote the direct sum of the representations $\rho _G$ and $\rho _{reg}$ by $\rho _{G}\oplus \rho_{reg}$. We slightly change the notations here, and declare a finite subgroup $\Ga \subset G$ to be sufficiently general if every $(\rho _{G}\oplus \rho_{reg})(\Ga)$-invariant subspace is also $(\rho _{G}\oplus \rho_{reg})(G)$-invariant. The proof of Proposition \ref{prop:big.is.general} holds also for the representation $\rho_G\oplus\rho_{reg}$ (maybe after changing the value of $q_0$). The rest of the claims quoted from \cite{LP} remain trivially true. In particular, in this new notation, we get the following
\end{rem}

\begin{prop} \lbl{prop:s.s.regular.element} Let $G$ be a simple, connected and adjoint algebraic group. Suppose that $\De \subset G$ is sufficiently general. Then for every $x\in G$, the set $x\De$ contains a regular semisimple element.
\end{prop}

\begin{defn} For a group $H$ and an element $h\in H$, we denote the centralizer of $h$ in $H$, i.e. the set of elements of $H$ that commute with $h$, by $C_H(h)$.
\end{defn}

\begin{prop} \lbl{prop:size.centralizer} \cite[Theorem 6.2]{LP}. Let $G$ be a simple, connected and adjoint algebraic group. There is a constant $C$ such that for any sufficiently general subgroup $\De \subset G$ and for any element $g\in \De$,
\[
\frac{1}{C}|\bF _\De| ^{\dim(C_G(g))} \leq |C_{\De}(g)| \leq C |\bF _\De| ^{\dim(C_G(g))}.
\]
\end{prop}

\begin{prop} \lbl{prop:centralizer.regular} Let $G$ be a simple, connected and adjoint algebraic group. If $\De_1,\De_2 \subset G$ are sufficiently general subgroups of $G$, if $|\bF _{\De_1}|$ and $|\bF _{\De_2}|$ are large enough, and if $w\in \De_1 \cap \De_2$ is a regular semisimple element, then $|C_{\De_1}(w)| \geq |C_{\De_2}(w)|$ implies $|\bF _{\De_1}| \geq |\bF _{\De_2}|$.
\end{prop}

\begin{proof} By Proposition \ref{prop:size.centralizer}, there is a constant $C$ such that if $\De$ is sufficiently general then for every element $g\in \Ga$,
\[
\frac{1}{C} |\bF _{\De}|^{\dim C_G(g)} \leq |C_{\De}(g)| \leq C |\bF _{\De}|^{\dim C_G(g)} .
\]
As $w$ is regular semisimple, $\dim C_G(w) =\rk G$. Therefore, if $|\bF _\De |$ is large enough we have
\[
|\bF _\De |=[ |C_\De (w)|^{1/\rk G} ],
\]
where $[x]$ denotes the integer closest to the real number $x$. The Proposition follows immediately from this.
\end{proof}

\section{A Representation Theoretic Lemma}

\begin{lem} \lbl{lem:invariant.line} Let $G$ be an algebraic group over a field of characteristic $p$. Let $\rho :G\to \GL_n(\overline{\bF _p})$ be a rational representation. For every subset $T\subset G$, there are $n^2$ elements $g_1,\ldots ,g_{n^2} \in T$ such that every line that is invariant under $\rho (g_1),\ldots ,\rho (g_{n^2})$ is also $\rho (T)$-invariant.
\end{lem}

\begin{proof} Given a subset $S\subset G$, let $U_1(S),\ldots ,U_k(S) \subset V$ be the common eigenspaces of $\rho (S)$ (i.e. for each $s\in S$, $\rho (s)$ acts as a scalar on each $U_i(S)$ and the $U_i(S)$ are maximal with respect to this property). Define
\[
w(S)=\sum _{i=1} ^{k} \dim U_i(S)^2 .
\]
For example, $w(\emptyset )=(\dim V)^2$, and $w(S)=0$ if and only if $\rho (S)$ does not have invariant lines. As a function of $S$, $w(S)$ is monotonic non-increasing and for $S\subset T$, $w(S)=w(T)$ if and only if every line invariant under $\rho (S)$ is invariant under $\rho (T)$.

We show now that for each $S$ with $w(S)>w(T)$ we can find $g\in T$ such that $w(S\cup \{ g\} )<w(S)$. Hence, starting from the empty set and adding no more than $n^2$ elements we arrive to a set that has the same invariant lines as $T$.

Suppose that $S$ is given and let $U_1(S),\ldots ,U_k(S)$ be its eigenspaces. Since $w(S)>w(T)$, there is a line $\ell$ invariant under $\rho (S)$ but not under $\rho (T)$. We can assume without loss of generality that $\ell \subset U_1(S)$. Since $\ell$ is not $\rho (T)$-invariant, there is $g\in T$ such that $\ell$ is not preserved under $\rho (g)$. Note that each eigenspace of $S\cup \{ g\}$ is contained in some $U_i(S)$. Let $W_{1,1}(S\cup \{ g\} ),\ldots,W_{1,m_1}(S\cup \{ g\} ),\ldots,W_{k,m_k}(S\cup \{ g\} )$ be the eigenspaces of $S\cup \{ g\}$ ordered such that $W_{i,j}(S\cup \{ g\} ) \subset U_i(S)$. Since for nonnegative $a$ and $b$ we have $a^2+b^2 \leq (a+b)^2$ and equality occurs if and only if one of $a,b$ are 0, we have
\[
w(S\cup \{ g\} )=\sum _i \sum _j \dim W_{i,j} (S\cup \{ g\} )^2 < \sum _i \dim U_i(S)^2 =w(S).
\]
\end{proof}

\begin{lem} \lbl{lem:invariant.subspace} Let $G$ be an algebraic group over a field of characteristic $p$. Let $\rho :G\to \GL_n(\overline{\bF _p})$ be a rational representation. For every subset $T\subset G$, there are $2^{n^2}$ elements $g_1,\ldots ,g_{2^{n^2}} \in T$ such that every subspace invariant under $\rho (g_1),\ldots ,\rho (g_{2^{n^2}})$ is already invariant under $\rho (T)$.
\end{lem}

\begin{proof} Apply Lemma \ref{lem:invariant.line} for the representation $V\oplus \wedge ^2 V \oplus \cdots \oplus \wedge ^{\dim V} V$.
\end{proof}

\begin{prop} \lbl{prop:subset.to.suff.general} Let $G$ be a simple, connected and adjoint algebraic group. For every positive integer $n$, there is a constant $c_1=c_1(\rk G,n)$ that depends only on $n$ and the Lie rank of $G$ such that the following is true: For every $k\geq c_1$, if $g_1,\ldots ,g_k \in G$ and the subgroup $\Ga$ generated by $g_1,\ldots ,g_k$ is sufficiently general and has more than $n$ elements, then there is a subset $T \subset \{ g_1 ,\ldots ,g_k\}$ of size $c_1$ such that every finite subgroup of $G$ that contains $T$ is sufficiently general and has more than $n$ elements.
\end{prop}

\begin{proof} By definition, a subgroup $\De \subset G$ is sufficiently general if and only if $\rho_G (\De )$ has the same invariant subspaces as $\rho_G (G)$, where $\rho _G$ is the representation in \S\ref{defn.suff.gen}. Let $c_1=2^{(\dim \rho _G)^2}+\log_2n$. By Lemma \ref{lem:invariant.subspace}, there is a subset $T\subset \{ x_1 ,\ldots ,x_k\}$ of size $2^{(\dim \rho _G)^2}$ such that every $\rho (T)$-invariant subspace is also $\rho (\{ x_1,\ldots ,x_k\})$-invariant, and hence $\rho(G)$-invariant. Therefore the subgroup generated by $T$ is sufficiently general. After adding no more than $\log_2 n$ additional elements, the subgroup generated by $T$ has more than $n$ elements.
\end{proof}

\section{Proof of Theorem \ref{thm:main.theorem}}
Let $\Ga$ be a finite, quasi-simple group of Lie type with rank at most $r$, defined over the field of characteristic $p\neq2,3$.

Our aim is to show the existence of a function $c=c(r)$, depending only on the Lie rank of $\Ga$, such that the extended PRA graph $\tilde X_k(\Ga)$ is connected for all $k \geq c$.

The next Lemma is a special case of the Gasch\"{u}tz Lemma~\cite{Gas} (see also \cite[Lemma 2.1.5]{P}). For the convenience of the reader, we give a proof for our case.

\begin{lem} \lbl{lem:Shelly.will.not.give.german.references.again} Let $K$ be a finite abelian group which is generated by $n$ elements. Suppose that $a,b_1,\ldots ,b_n \in K$. Then there are integers $m_1,\ldots,m_n$ such that
\[
\langle a,b_1,\ldots,b_n \rangle = \langle a^{m_1}\cdot b_1 ,\ldots, a^{m_n}\cdot b_n \rangle .
\]
\end{lem}

\begin{proof} Denote the group $\langle a,b_1,\ldots,b_n \rangle$ by $L$. Since $L\subset K$ are abelian groups, the group $L$ is generated by $n$ (or fewer) elements. We prove the Lemma in the following cases:

{\bf Case 1.} {\em $L \subset \bF_p ^d$, for some integer $d$ and prime number $p$:} By our assumption $d\leq n$. If the $b_i$'s are linearly independent over $\bF_p$, we can take $m_i=0$ for all $i$. If the $b_i$'s are dependent, there is $i_0$ such that $b_{i_0}$ is in the span of the rest of the $b_j$'s. In this case let $m_j=0$ if $j\neq i_0$, and let $m_{i_0}=1$.

{\bf Case 2.} {\em $L \subset \bF_{p_1}^{d_1} \times \ldots \times \bF_{p_k}^{d_k}$, for some $k>0$, different primes $p_1,\ldots,p_k$, and integers $d_i$:} By our assumption, $d_i\leq n$, thus by the previous case, there are integers $m_{i,j}$, such that for every $j=1,\dots,k$, the images of the groups
\[
\langle a,b_1,\ldots,b_n \rangle \quad \textrm{and} \quad \langle a^{m_{1,j}}\cdot b_1,\ldots, a^{m_{n,j}}\cdot b_n \rangle
\]
in $\bF_{p_j}^{d_j}$ are equal. By the Chinese Remainder Theorem, there are integers $m_i$ such that for all $j=1,\ldots,k$,
\[
m_i\equiv m_{i,j}\quad \textrm{(mod $p_j$)}.
\]
We get that the group $\langle a,b_1,\ldots,b_n \rangle$ has the same image in $\bF_{p_j}^{d_j}$ as the group $\langle a^{m_1}\cdot b_1,\ldots, a^{m_n}\cdot b_n \rangle$. Since the $p_j$'s are different primes, the groups themselves are equal.

{\bf Case 3.} {\em The general case:} Let $\Phi(L)$ denote the Frattini subgroup of $L$, i.e. the intersection of all maximal proper subgroups. It is easily seen that $L/\Phi(L)$ is a product of elementary abelian groups. By Case 2. above, there are integers $m_i$ such that the subgroup generated by
\[
a^{m_1}\cdot b_1\Phi(L),\ldots,a^{m_n}\cdot b_n\Phi(L)
\]
is equal to $L/\Phi(L)$. This, however, implies that the subgroup generated by $a^{m_1}\cdot b_1,\ldots,a^{m_n}\cdot b_n$ is equal to $L$.
\end{proof}

\subsection{Getting a Redundant Element} A generating tuple $(x_1,\ldots,x_k)$ for $\Ga$ is called {\em redundant}, if there is some $i$ such that the group generated by $\{ x_j\}_{j\neq i}$ is equal to $\Ga$.

\begin{prop} \lbl{prop:get.redundant.element} Let $G$ be a simple, connected and adjoint algebraic group. Let $\Ga \subset G$ be a sufficiently general finite subgroup. There is a constant $c=c(\rk(G))$ that depends only on the Lie rank of $G$ such that for every $k \geq c$, if $(x_1,\ldots ,x_k) \in \Ga$ is a generating $k$-tuple, then one can perform Nielsen moves on $(x_1,\ldots ,x_k)$, such that the resulting generating tuple is redundant.
\end{prop}

\begin{proof} Fix $n$ to be large enough with respect to $\rk(G)$. Let $c_1=c_1(\rk(G),n)$ be the constant of Proposition \ref{prop:subset.to.suff.general} and let $a_0=a_0(\rk(G))$ be the constant from Proposition \ref{prop:conj.classes}. The constant $c$ is defined as
\[
c=c_1+2a_0\cdot \rk(G)+1 .
\]

Suppose $k \geq c$. Let $\Ga$ be a sufficiently general subgroup, and let $x_1,\ldots ,x_k \in \Ga$ generate $\Ga$. By Proposition \ref{prop:subset.to.suff.general}, after reordering the $x_i$'s, we can assume that $x_1,\ldots ,x_{c_1}$ generate a sufficiently general subgroup of $G$. We denote this subgroup by $\De _0$. Since $n$ is arbitrary large, we can assume that $|\De_0|$, and hence $|\bF_{\De_0}|$, is large enough. By Proposition \ref{prop:s.s.regular.element}, the set $x_{c_1+1}\De _0$ contains a regular semisimple element. After applying Nielsen moves, we can assume both that $x_1,\ldots ,x_{c_1}$ generate the sufficiently general subgroup $\De _0$, and that $x_{c_1+1}$ is regular and semisimple.

We make the following notations: denote $c_2=a_0\cdot \rk(G)$, denote $w=x_{c_1+1}$, and for $1\leq i \leq c_2-1$ denote $y_i=x_{c_1+1+i}$ and $z_i=x_{c_1+1+c_2+i}$. Let $\De _1$ be the subgroup generated by $x_1,\ldots ,x_{c_1},w,y_1,\ldots ,y_{c_2}$, and let $\De _2$ be the subgroup generated by $x_1,\ldots ,x_{c_1},w,z_1,\ldots ,z_{c_2}$.

Since they both contain $\De _0$, the subgroups $\De_1$ and $\De_2$ are sufficiently general. By Proposition \ref{prop:pair.suff.gen}, there are Frobenius maps $F_1$ and $F_2$, such that $(G^{F_1})^{\der} \subset \De_1 \subset G^{F_1}$ and $(G^{F_2})^{\der} \subset \De_2 \subset G^{F_2}$. Moreover, $F_1$ and $F_2$ are powers of $F_0$. We divide the rest of the proof into the following two cases:

{\bf Case 1.} $\bF _{\De_1}=\bF_{\De_2}$: From Proposition
\ref{prop:pair.suff.gen} we get that $F_1 = F_2$, and so both
$\De_1$ and $\De_2$ contain $(G^{F_1})^{\der}$ and are contained in
$G^{F_1}$. As $|G^{F_1}/(G^{F_1})^{\der}|\leq 2\rk(G)$, by Remark \ref{rem:schur.mult}, we get that $|\langle \De_1,\De_2 \rangle / (G^{F_1})^{\der}|\leq 2\rk(G)$. Therefore there are $m=\lceil \log_2(\rk(G))\rceil+1$ generators
$z_{i_1},\ldots,z_{i_m}$ such that
\[
\Ga = \langle \De_1,\De_2 \rangle = \langle \De_1,z_{i_1},\ldots,z_{i_m} \rangle .
\]
Since $c_2>\lceil \log_2(\rk(G))\rceil+1$, the tuple $(x_1,\ldots ,x_{c_1},w,y_1,\ldots ,y_{c_2},z_1,\ldots ,z_{c_2})$ is redundant.

{\bf Case 2.} $|\bF _{\De_1}| \neq |\bF_{\De_2}|$: As $\bF_{\De _1},\bF _{\De _2} \subset \bF _\Ga$, the possible values of $|\bF_{\De_1}|+|\bF_{\De_2}|$ are bounded. Finitely many applications of the next Lemma will give us a sequence of Nielsen moves such that the resulting tuple is in Case 1., for which the claim was already proved.
\end{proof}

\begin{lem} \lbl{lem:enlarge.defn.field.} Let $\De_1, \De_2,\Ga,x_i,w,y_i,z_i$ be as in the proof of the Proposition above. If $\bF _{\De_1} \neq \bF _{\De_2}$ then there are Nielsen moves one can apply to $\{ x_i,w,y_i,z_i \}$ such that the new value of $|\bF _{\De_1}|+|\bF _{\De_2}|$ is strictly larger than the old one.
\end{lem}

In order to facilitate the reading of the proof, we first present the argument in a special case, following \cite{E}. The core of the argument is the following claim: Suppose that $G=\PGL_2$, that $w$ is diagonal and not the identity, and that $y\in G$. Assume that $w$ is conjugate to $w^y$ inside the group $\langle w,w^y\rangle$. Then for every $q$, there is an element $\al\in\bF_q^\times$ such that the group generated by $w$ and $\bmat\al&0\\0&\al^{-1}\emat y$ contains all diagonal matrices with coefficients in $\bF_q^\times$. In particular, if this group is sufficiently general, then it must contain $\PGL_2(\bF_q)$.

The proof of the claim is simple: Suppose that $d\in\langle w,w^y\rangle$ satisfies that $w^d=w^y$. Then $yd^{-1}$ commutes with $w$, and hence is diagonal, i.e. $yd^{-1}=\bmat \be&0\\0&\be^{-1}\emat$. Take $\al\in\bF_q^\times$ such that $\al\be$ is a generator of $\bF_q^\times$. Since both $w$ and $w^y$ are contained in the group generated by $w$ and $\bmat\al&0\\0&\al^{-1}\emat y$, it follows that this last group contains the element $\bmat\al&0\\0&\al^{-1}\emat yd^{-1}=\bmat \al\be&0\\0&\al^{-1}\be^{-1}\emat$, and hence all diagonal elements with coefficients in $\bF_q$.

Going back to the lemma, assume that the group $\langle \De_0,w,z_i\rangle$ is of the form $\PGL_2(\bF_q)$. Let $y_i'=\bmat \al&0\\0&\al^{-1}\emat y_i$. There are Nielsen moves that take $\{x_i,w,y_i,z_i\}$ to $\{x_i,w,y_i',z_i\}$, but now $\langle x_i,w,y_i'\rangle$ contains $\PGL_2(\bF_q)$, so $|\bF_{\langle x_i,w,y_i'\rangle}|\geq|\bF_{\langle x_i,w,z_i\rangle}|$.

\begin{proof}
Recall that $\bF_{\De _1},\bF _{\De _2} \subset \bF _\Ga$. Without loss of generality we may assume that $|\bF _{\De_1}| < |\bF _{\De_2}|$.
Define
\[
\De_1'=\langle \De_0,w^{y_1},\ldots ,w^{y_{c_2}}\rangle,
\]
and note that $\De_1'$ is sufficiently general.

Denote $r = \rk(G)$ and recall that $c_2=a_0 \cdot r$. Consider the
elements $w^{y_i}=(y_i)^{-1}wy_i$, and note that all of them are
conjugate in $G$. By Proposition \ref{prop:conj.classes}, there are
at least $r+1$ of the $w^{y_i}$'s that are pairwise conjugate in
$\De_1'$. Without loss of generality, we may assume that
$w^{y_1},\ldots,w^{y_r}$ are all conjugate to $w^{y_{c_2}}$ in
$\De_1'$. Spelling this out, there are $d_1,\ldots,d_r \in \De_1'$
such that
\[
w^{y_i} = w^{y_{c_2}d_i} \quad \textrm{for any $i=1,\ldots,r$}.
\]

Define $u_i = y_i d_i^{-1} y_{c_2}^{-1}$ for $i=1,\dots,r$, and note
that $u_i \in C_{\De_1}(w)$.

Since $w$ is a semisimple regular element,
$C_{\Ga}(w)$ is a product of at most $r$ cyclic subgroups. Thus we
may assume that $C_{\De_2}(w) = \langle c_1,\dots,c_r \rangle$.
Since $c_1,\dots,c_r,u_1,\dots,u_r \in C_{\Ga}(w)$, by repeated application of
Lemma \ref{lem:Shelly.will.not.give.german.references.again},
we can find integers $m_{i,j}$ $(1 \leq i,j \leq r)$ such that
\begin{align*}
    K &= \langle c_1,\dots,c_r,u_1,\dots,u_r \rangle =
    \\
    &= \langle c_1^{m_{1,1}}\dots c_r^{m_{1,r}} u_1,
    c_1^{m_{2,1}}\dots c_r^{m_{2,r}} u_2,
    \ldots, c_1^{m_{r,1}}\dots c_r^{m_{r,r}} u_r\rangle.
\end{align*}

For every $i=1,\dots,r$, denote $b_i = c_1^{m_{i,1}}\dots
c_r^{m_{i,r}}$. Since $c_1,\dots,c_r$ are in $\De_2$ we can
connect by Nielsen moves
\begin{align*}
    &(x_1,\dots,x_{c_1},w,y_1,\dots,y_r,\dots,y_{c_2},z_1,\dots,z_{c_2})
    \\
    \rightarrow &(x_1,\dots,x_{c_1},w,b_1y_1,\dots,b_{r}y_{r},y_{r+1}\dots,y_{c_2},z_1,\dots,z_{c_2})
\end{align*}

Let $\De_3 = \langle \De_0,w,b_1y_1,\dots,b_ry_r,y_{r+1},\dots,y_{c_2}\rangle$.
By Proposition \ref{prop:pair.suff.gen},
there is a Frobenius map $F_3$, such that $(G^{F_3})^{\der} \subset
\De_3 \subset G^{F_3}$. Moreover, $\De_3 \supset \De_1'$, since

\[
w^{b_iy_i} = w^{c_1^{m_{i,1}}\dots c_r^{m_{i,r}}y_i} = w^{y_i}.
\]

Therefore, $d_i \in \De_3$ and the following elements are also in
$\De_3$:
\[
    b_i y_i d_i^{-1} y_{c_2}^{-1} = b_i u_i \quad (i=1,\dots,r)
\]

However, these are exactly the generators of the group $K$ defined above.
Therefore $c_1,\dots,c_r$ actually belong to $\De_3$, and we
deduce that $C_{\De_3}(w) \geq C_{\De_2}(w)$. By Proposition
\ref{prop:centralizer.regular}, $|\mathbb{F}_{\De_3}| \geq
|\mathbb{F}_{\De_2}|$, and therefore
\[
|\mathbb{F}_{\De_3}|+|\mathbb{F}_{\De_2}| > |\mathbb{F}_{\De_1}|+|\mathbb{F}_{\De_2}|.
\]

\end{proof}

\begin{cor} \lbl{cor:get.redundant} Let $\Ga$ be a finite quasi-simple group of Lie type and rank $r$. There is a constant, $c=c(r)$, that depends only on $r$, such that for every $k \geq c$, if $(x_1,\ldots ,x_k)$ is a generating $k$-tuple of $\Ga$, then one can apply Nielsen moves on $(x_1\ldots,x_k)$ to get a redundant tuple.
\end{cor}

\begin{proof} The claim is trivial for $\Ga$ of bounded size. Indeed, if $k>\log_2(|\Ga|)$, then every generating $k$-tuple is redundant. For a quasi-simple group $\Ga$ of rank $r$, the size of the center $Z(\Ga)$ is bounded by $r$. Therefore, the claim is trivially true if $\Ga/Z(\Ga)$ has bounded size.

The group $\Ga/Z(\Ga)$ is a simple group of Lie type. It is therefore of the form $(G^F)^{\der}$ for a simple, connected and adjoint group $G$, and a Frobenius map $F:G\to G$. Since the rank of $G$ is bounded, we get that $q_F$ tends to infinity as $|(G^F)^{\der}|$ tends to infinity (see Remark \ref{rem:size.field}). Therefore, we may assume that $q_F$ is large enough, and, by Proposition \ref{prop:big.is.general}, that $(G^F)^{\der}$ is sufficiently general.

Let $c=c(r)$ be the constant of Proposition \ref{prop:get.redundant.element}, let $k>c+\log_2 r$, and let $(x_1,\ldots,x_k)$ be a generating tuple of $\Ga$. Denote the images of $x_i$ in $\Ga/Z(\Ga)$ by $\overline{x_i}$. By Proposition \ref{prop:get.redundant.element}, after applying Nielsen moves we obtain a generating tuple $(y_1,\ldots,y_k)$ such that $\overline{y_1},\ldots,\overline{y_c}$ generate $\Ga/Z(\Ga)$. This implies that $y_1,\ldots,y_c$ generate $\Ga$: Let $H$ be the subgroup generated by $y_1,\ldots,y_c$. Then $\Ga=H\cdot Z(\Ga)$, and hence $H$ is normal in $\Ga$ and the quotient $\Ga/H$ is abelian. By perfectness, $H=\Ga$.
\end{proof}

\subsection{Proof of Theorem \ref{thm:main.theorem}}

Let $\Ga$ be a quasi-simple group of Lie type and assume that $\Ga$ is generated by the elements $x_1 ,\ldots ,x_k$. Corollary \ref{cor:get.redundant} shows that there is a
constant $c = c(r)$, which depends only on the Lie rank of $\Ga$, such that
if $k \geq c$, then there are Nielsen moves that can be applied to
connect the generating tuple $(x_1,\ldots,x_k)$ to a redundant
tuple. By replacing $c$ by $c+1$, for every generating $k$-tuple
(where $k\geq c+1$), one can apply Nielsen moves and transform
$(x_1,\ldots ,x_k)$ to a generating tuple with two redundant
elements. This can be done by first applying some Nielsen moves that
yield one redundant generator, say $x_k$, and then since $\Ga =
\langle x_1,\dots,x_{k-1} \rangle$ (where $k-1 \geq c$), one can
apply some more Nielsen moves to get a second redundant generator.

Every finite quasi-simple group is generated by two elements. Let $\ga _1,\ga _2 \in \Ga$ generate $\Ga$. By the last paragraph, we can connect
\[
(x_1,\ldots ,x_k)
\]
to
\[
(y_1,\ldots ,y_{k-2},1,1)
\]
which can be connected to
\[
(y_1,\ldots ,y_{k-2},\ga _1,\ga _2)
\]
and thus to
\[
(1,\ldots ,1,\ga _1,\ga _2).
\]
Therefore every generating $k$-tuple can be connected to $(1,\ldots ,1,\ga _1,\ga _2)$, so $\widetilde{X}_k(\Ga)$ is connected.

\section{Appendix}
In this appendix we describe the adaptation of the proof of Theorem \ref{thm:main.theorem} in the case when the characteristic $p$ is equal to 2 or 3.

The only difference in small characteristics is the definition of the field $\bF_\Ga$ for a sufficiently general subgroup $\Ga$. The definition of $\bF_\Ga$ in this case is as follows: One looks at a Borel subgroup $B$ and a $B$-invariant connected algebraic subgroup $V$ of the center $Z(B)$, such that $\Ga\cap V\neq\{1\}$. We assume that $\dim(V)$ is minimal for all possible such $B$ and $V$. Then it is shown that the subring of $\End(V)$, which is generated by the normalizer of $V$ in $\Ga$, is indeed a field, and is independent of the choice of $B$ and $V$. This field is defined to be $\bF_\Ga$. All the theorems of \cite{LP} that were quoted in this paper are true for this new definition.

If $p>3$, then for any Borel subgroup, the center $Z(B)$ is one dimensional, and we recover the definition of $\bF_\Ga$ that was given in this paper. However, in characteristics 2 and 3 it might happen that $Z(B)$ is two dimensional. This makes our proof of Proposition \ref{prop:pair.suff.gen} incorrect. Proposition \ref{prop:pair.suff.gen} is used in the proofs of Proposition \ref{prop:get.redundant.element} and Lemma \ref{lem:enlarge.defn.field.}. We show now how to amend the proofs of those propositions.

Let $\de(\Ga)$ be the minimum of $\dim(V)$ for all connected algebraic groups $V$ that are contained in the center of a Borel subgroup of $G$, such that $\Ga\cap V\neq\{1\}$. Thus, $\de(\Ga)$ is either equals $1$ or $2$. The proof of Proposition \ref{prop:pair.suff.gen} shows that it remains true for characteristics 2 and 3, if we assume in addition that $\de(\De)=\de(\De')$. Also, Lemma \ref{lem:enlarge.defn.field.} remains true for characteristics 2 and 3 if we assume in addition that $\de(\De_1)=\de(\De_2)=\de(\Ga)$.

Lastly, we show that Proposition \ref{prop:get.redundant.element} remains true as stated in characteristics 2 and 3 (but the constant $c$ changes). We let $c=c_1+4\cdot a_0 \cdot \rk(G)+1$. As in the proof of Proposition \ref{prop:get.redundant.element}, we can assume that $\De_0=\langle x_1,\ldots,x_{c_1}\rangle$ is sufficiently general, and that $x_{c_1+1}$ is regular semisimple. Let
\[
y_i=x_{c_1+1+i} \qq z_i=x_{c_1+1+c_2+i} \qq \widetilde{y_i}=x_{c_1+1+2c_2+i} \qq \widetilde{z_i}=x_{c_1+1+3c_2+i}.
\]
and let $\widetilde{\De_0}=\langle \De_0,x_{c_1+1},y_1,\ldots,y_{c_2},z_1,\ldots,z_{c_2}\rangle$.

If $\de(\widetilde{\De_0})=2$ then for every subgroup $\Upsilon$ of $\widetilde{\De_0}$, we have that $\de(\Upsilon)=2$, and the proof of \ref{prop:get.redundant.element} shows that one can apply Nielsen transformations on the tuple
\[
x_1,\ldots,x_{c_1},x_{c_1+1},y_1,\ldots,y_{c_2},z_1,\ldots,z_{c_2}
\]
and get a redundant tuple. Hence, the same is true for the tuple $x_1,\ldots,x_k$.

On the other hand, if $\de(\widetilde{\De_0})=1$, then for every subgroup $\Upsilon$ of $G$ that contains $\widetilde{\De_0}$, we have that $\de(\Upsilon)=1$. The proof of Proposition \ref{prop:get.redundant.element}, applied to $\widetilde{\De_0},\widetilde{y_i},\widetilde{z_i}$ instead of $\De_0,y_i,z_i$ shows that one can apply Nielsen moves to $x_1,\ldots,x_k$ and get a redundant tuple.

%%%%%%%%%%%%%%%%%%%%%%%%%%%%%%%%%%%%%%%%%%%%%%%%

\vspace{\bigskipamount}

\begin{footnotesize}
\begin{quote}

Nir Avni\\
Einstein Institute of Mathematics,\\
The Hebrew University of Jerusalem, \\
Edmond Safra Campus, Givat Ram, \\
 Jerusalem 91904, Israel \\
{\tt  avni.nir@gmail.com} \\

Shelly Garion\\
Einstein Institute of Mathematics, \\
The Hebrew University of Jerusalem, \\
Edmond Safra Campus, Givat Ram, \\
 Jerusalem 91904, Israel \\
{\tt shellyg@math.huji.ac.il} \\

\end{quote}
\end{footnotesize}


\addcontentsline{toc}{section}{References}

\begin{thebibliography}{MMMM}

%\bibitem{At}
%J.H. Conway, R.T. Curtis, S.P. Norton, R.A. Parker, R.A. Wilson,
%Atlas of finite groups. Maximal subgroups and ordinary characters
%for simple groups. With computational assistance from J. G.
%Thackray. {\it Oxford University Press}, Eynsham (1985).


\bibitem{BP} L. Babai, I. Pak, Strong bias of group generators: an
obstacle to the "product replacement algorithm", {\it Proc. Eleventh
Annual ACM-SIAM Simposium on Discrete Algorithms} (2000).

%\bibitem{BGK}
%T. Breuer, R.M. Guralnick, W.M. Kantor, Probabilistic generation of finite simple groups II, preprint, 2006.

\bibitem{Ca2}
R.W. Carter, Finite groups of Lie type. Conjugacy classes and
complex characters. Pure and Applied Mathematics (New York). A
Wiley-Interscience Publication. {\em John Wiley and Sons, Inc., New
York}, 1985.

\bibitem{Ca1}
R.W. Carter, Simple groups of Lie type.
Pure and Applied Mathematics (New York). A Wiley-Interscience Publication.
{\em John Wiley and Sons, Inc., New York}, 1972.

\bibitem{CLMNO}
F. Celler, C.R. Leedham-Green, S. Murray, A. Niemeyer, E.A. O'Brien,
Generating random elements of a finite group, \textit{Comm. Alg.}
\textbf{23} (1995), 4931--4948.

\bibitem{CP}
G. Cooperman, I. Pak, The product replacement graph on generating
triples of permutations, preprint, 2000.

\bibitem{Da}
C. David, $T_3$ systems of finite simple groups, \textit{Rend. Sem.
Math. Univ. Padova} \textbf{89} (1993), 19--27.

\bibitem{Du1}
M.J. Dunwoody, On $T$-systems of groups, \textit{J. Austral. Math.
Soc.} \textbf{3} (1963), 172--179.

\bibitem{Du2}
M.J. Dunwoody, Nielsen transformations, \textit{in: Computation
Problems in Abstract Algebra}, Pergamon, Oxford, 1970, 45--46.

\bibitem{E}
M.J. Evans, $T$-systems of certain finite simple groups,
\textit{Math. Proc. Cambridge Philos. Soc.} \textbf{113} (1993),
9--22.

\bibitem{GaP} A. Gamburd, I. Pak,
Expansion of product replacement graphs, {\it Combinatorica} {\bf
26} (2006), no. 4, 411--429.

\bibitem{Ga} S. Garion, Connectivity of the Product Replacement Algorithm Graph of
$\bf{\PSL(2,q)}$ (2007), preprint.

\bibitem{GS} S. Garion, A. Shalev, Commutator maps, measure
preservation, and $T$-systems, \textit{to appear in Trans. Amer.
Math. Soc.}

\bibitem{Gas} W. Gasch\"{u}tz, Zu einem von B. H. und H. Neumann gestellten
problem, {\it Math. Nachr.} {\bf 14} (1955), 249--252.

\bibitem{Gi}
R. Gilman, Finite quotients of the automorphism group of a free
group, \textit{Canad. J. Math.} \textbf{29} (1977), 541--551.

%\bibitem{GK}
%R.M. Guralnick, W.M. Kantor, Probabilistic generation of finite simple groups, \textit{J. Algebra}, \textbf{234} (2000), 743--792.

\bibitem{GP} R.M. Guralnick, I. Pak, On a question of B.H.
Neumann, \textit{Proc. Amer. Math. Soc.} \textbf{131} (2003),
2001--2025.

\bibitem{LP1}
A. Lubotzky, I. Pak, The product replacement algorithm and Kazhdan's
property (T), {\it Journal of the AMS } {\bf 52} (2000), 5525--5561.

\bibitem{LP}
M. J. Larsen and R. Pink, Finite subgroups of Algebraic groups
(1998), preprint.

\bibitem{Ne}
B.H. Neumann, On a question of Gasch\"utz, \textit{Arch. Math.}
\textbf{7} (1956), 87--90.

\bibitem{NN}
B.H. Neumann, H. Neumann, Zwei klassen charakteristischer
untergruppen und ihre faktorgruppen, \textit{Math. Nachr.}
\textbf{4} (1951), 106--125.

\bibitem{Ni}
N. Nikolov, On the size of independent sets of generators for finite
simple groups of Lie type (2005), preprint.

\bibitem{P}
I. Pak, What do we know about the product replacement algorithm?,
\textit{in: Groups and computation, III}, de Gruyter, Berlin (2001),
301--347.

\end{thebibliography}
\end{document}